\documentclass[10pt]{amsart}
\usepackage{url}
\usepackage{hyperref}
\usepackage{enumerate}
\usepackage{comment}

\makeatletter

\@namedef{subjclassname@2010}{

  \textup{2010} Mathematics Subject Classification}

\usepackage{amssymb, amsmath,amsthm}
\usepackage{mathrsfs}
\newtheorem{thm}{Theorem}[section]

\newtheorem{prop}[thm]{Proposition}

\newtheorem{cor}[thm]{Corollary}

\newtheorem{lem}[thm]{Lemma}
\theoremstyle{definition}
\newtheorem{rem}{Remark}

\newcommand{\ra}{\rightarrow}

\newcommand{\mc}{\mathcal}

\newcommand{\mb}{\mathbb}
\newcommand{\sg}{\sigma}
\newcommand{\eps}{\epsilon}

\newcommand{\const}{1/4}
\renewcommand{\ss}{\substack}

\newcommand{\e}{\varepsilon}

\newcommand{\asum}{\sideset{}{^{\ast}}\sum}
\renewcommand{\bar}{\overline}

\begin{document}
\title{On a rigidity property for Quadratic Gauss Sums}
\author{Alexander P. Mangerel}
\address{Department of Mathematical Sciences, Durham University, Stockton Road, Durham, DH1 3LE, UK}
\email{smangerel@gmail.com}
\begin{abstract}
Let $N$ be a large prime and let $c > \const$. We prove that if $f$ is a $\pm 1$-valued multiplicative function, such that the exponential sums 
$$
S_f(a) := \sum_{1 \leq n < N} f(n) e(na/N), \quad a \pmod{N}
$$
satisfy the ``Gauss sum-like'' approximate dilation symmetry property 
$$
\frac{1}{N}\sum_{a \pmod{N}} |S_f(ap) - f(p)S_f(a)|^2 = o(N),
$$
uniformly over all primes $p \leq N^c$ then $f$ coincides with a real character modulo $N$ at \emph{almost all} $n < N$. As a consequence, taking $f$ to be the Liouville function we connect this exponential sums property to the location of real zeros of $L(s,\chi)$ close to $s = 1$, for $\chi$ the Legendre symbol modulo $N$. \\
Assuming the $L$-functions of primitive Dirichlet characters modulo $N$ have a sufficiently wide zero-free region (of Littlewood type), we also show a more general result in which any $c > 0$ may be taken. 
\end{abstract}
\maketitle
\section{Introduction and main results}
Throughout this paper, let $N$ be a large prime and let $\chi = \chi_N := (\tfrac{\cdot}{N})$ denote the Legendre symbol modulo $N$. Write\footnote{As usual, for $t \in \mb{R}$ we write $e(t) := e^{2\pi i t}$.}
$$
S_{\chi}(a) := \sum_{1 \leq n < N} \chi(n)e(an/N), \quad a \pmod{N}
$$
to denote the twisted Gauss sum of $\chi$. It is well-known that $S_{\chi}$ satisfies the dilation symmetry property
\begin{equation} \label{eq:GaussRigreg}
S_{\chi}(ab) = \chi(b) S_{\chi}(a) \text{ for all } a \pmod{N}
\end{equation}
and any $b \in \mb{N}$. Since $\chi$ is completely multiplicative, i.e., $\chi(mn) = \chi(m)\chi(n)$ for all $m,n \in \mb{N}$, this is equivalent to the (slightly weaker) property that
\begin{equation}\label{eq:GaussRig}
S_{\chi}(ap) = \chi(p) S_{\chi}(a) \text{ for all } a \pmod{N},
\end{equation}
and any \emph{prime} $p$. It is natural to ask to what extent this property of Gauss sums characterises quadratic Dirichlet characters among all $\pm 1$-valued\footnote{As Gauss sums are supported on the interval $\mb{Z} \cap [1,N-1]$, wherein $\chi(n) \neq 0$, we will think of $\chi$ in the sequel as non-vanishing; note that this model is obviously inappropriate for $N$ \emph{composite}.} multiplicative functions. Moreover, it is also natural to wonder whether the symmetry property \eqref{eq:GaussRig} can be relaxed somewhat, without admitting counterexamples that are not ``character-like'', in some sense.\\
To fix ideas, let $f: \mb{N} \ra \{-1,+1\}$ be a multiplicative function, and set
%
%
$$
S_f(a) = S_f(a;N) := \sum_{1 \leq n < N} f(n) e(an/N), \quad a \pmod{N}.
$$
In this note we address which such functions $f$ satisfy a weaker form of \eqref{eq:GaussRig}: roughly speaking, we wish to characterise those $f$ such that
\begin{equation}\label{eq:Sfdilprop}
S_f(ap) \sim g(p) S_f(a) \text{ on average over $a \pmod{N}$,}
\end{equation}
\emph{uniformly} over $p \leq P$ for some (sufficiently quickly-growing) function $P = P(N)$ and some sequence $(g(p))_{p \leq P}$ (see \eqref{eq:unifSf} below for the precise property). An approximate dilation symmetry property of this type arose in our recent paper \cite{ManShus}, in connection with a problem of Shusterman on sign patterns $(\lambda(n),\lambda(M-n))$ of the Liouville function, for large even integers $M$. \\
Heuristically, one might expect that if $f$ satisfies \eqref{eq:Sfdilprop} then, at least for \emph{typical} $p \leq N^c$, $g(p) = f(p)$ and $f$ must ``behave like'' a real character modulo $N$ within the fundamental domain $[1,N-1]$. 
We shall prove the following unconditional result that makes this heuristic precise.
\begin{thm} \label{thm:dil}
Let $c > \const$. 
Let $N$ be a large prime, and let $f: \mb{N} \ra \{-1,+1\}$ be multiplicative. Let $1 \leq M \leq N$, and suppose that
\begin{equation}\label{eq:unifSf}
\max_{p \leq N^c} \frac{1}{N} \sum_{a \pmod{N}} \left|S_f(ap) - g(p) S_f(a)\right|^2 \leq \frac{N}{M}
\end{equation}
for some $\pm 1$-valued sequence $(g(p))_{p \leq N^c}$. Then there is a real character $\psi$ modulo $N$
such that
\begin{equation*} 
|\{n < N : f(n) \neq \chi(n)\}| \ll \frac{N}{\sqrt{M}}.
\end{equation*}
Moreover, 
there is an absolute constant $M_0$ such that if \eqref{eq:unifSf} holds with $M \geq M_0$ then
$$
\sum_{\ss{p < N \\ f(p) \neq \psi(p)}} \frac{1}{p} \ll \frac{1}{\min\{M^{1/4}, \sqrt{\log N}\}},
$$
and $g(p) = f(p)$ whenever $f(p) = \psi(p)$.
\end{thm}
As an application we obtain the following simple corollary of Theorem \ref{thm:dil} when applied to the Liouville function $\lambda$.
\begin{cor}\label{cor:Liou}
Let $c > 1/4$ and let $N$ be a large prime. If there is a $\pm 1$-valued sequence $(g(p))_{p \leq N^c}$ such that as $N \ra \infty$,
\begin{equation}\label{eq:Lioucond}
\max_{p \leq N^c} \frac{1}{N} \sum_{a \pmod{N}} \left|S_\lambda(ap) -g(p)S_\lambda(a)\right|^2  = o(N),
\end{equation}
then the Dirichlet $L$-function $L(s,(\tfrac{\cdot}{N}))$ has a real zero $\beta$ with $0 < 1-\beta = O(\tfrac{1}{\log N})$. 
\end{cor}
In other words, 
if one could prove a zero-free region for $L(s,(\tfrac{\cdot}{N}))$ of the shape
$$
\{s = \sg + it : \, \sg \geq 1-\tfrac{C}{\log N}, \, |t| \leq N\}
$$
with any \emph{sufficiently large} constant $C > 0$ then \eqref{eq:Lioucond} is not possible. 
\subsection{A conditional extension}
If we are willing to assume that the Dirichlet $L$-functions modulo $N$ have a suitably large zero-free region then the constraint $c > 1/4$ may be relaxed, and we may obtain a necessary and sufficient characterisation of functions with the property \eqref{eq:unifSf}. However, rather than \eqref{eq:unifSf} holding for \emph{all} $p \leq N^c$, the required condition applies to primes outside of a \emph{very sparse} set. \\
In the sequel, given $A > 0$ we say that $N$ is of \emph{$A$-Littlewood type} if 
all of the $L$-functions associated to non-principal Dirichlet characters modulo $N$ are non-vanishing in the region
\begin{equation}\label{eq:LWZFR}
\left\{s \in \mb{C} : \text{Re}(s) > 1-\frac{A \log\log N}{\log N}, \, |\text{Im}(s)| \leq N\right\}.
\end{equation}
\begin{thm}\label{thm:iff}
Let $N$ be a large prime, let $1 \leq M \leq N$ and let $f: \mb{N} \ra \{-1,+1\}$ be multiplicative. \\
(a) If there is a real character $\psi \pmod{N}$ such that the estimate
\begin{equation}\label{eq:almostall}
|\{n < N : f(n) \neq \psi(n)\}| \ll \frac{N}{M}.
\end{equation}
holds then there is a set $\mc{A} = \mc{A}_f \subseteq \mb{P} \cap (0,N)$ and a $\pm 1$-valued sequence $(g(p))_{p < N}$ such that 
\begin{align*}
&(i)\, \max_{p \in \mc{A}} \frac{1}{N} \sum_{a \pmod{N}} \left|S_f(ap) - g(p)S_f(a)\right|^2 \ll \frac{N}{M}, \\
&(ii)\, \sum_{\ss{p < N \\ p \notin \mc{A}}} \frac{1}{p} \ll \frac{1}{\sqrt{\min\{M,\log N\}}}.
\end{align*} 
(b) Let $A > 0$ and $c \in (0,1)$. Assume that $A$ and $M$ are both sufficiently large with respect to $c$, and that $N$ is a large prime of $A$-Littlewood type. Finally, suppose there exists $\mc{A} \subseteq \mb{P} \cap (0,N)$ and a $\pm 1$-valued sequence $(g(p))_{p \leq N^c}$ such that 
\begin{align} \label{eq:genA}
&\max_{\ss{p \in \mc{A} \cap [2,N^c]}} \frac{1}{N} \sum_{a \pmod{N}} |S_f(ap)-g(p)S_f(a)|^2 \ll \frac{N}{M}, \\
&\sum_{\ss{p < N \\ p \notin \mc{A}}} \frac{1}{p} \ll \frac{1}{\sqrt{\min\{M,\log N\}}}.  \nonumber
\end{align}
Then there is a real character $\psi$ such that $|\{n < N : f(n) \neq \psi(n)\}| \ll N/\sqrt{M}$.
\end{thm}
\subsection{Remarks on the main results}
\begin{rem}
By Plancherel's theorem $\pmod{N}$ and the Cauchy-Schwarz inequality, the trivial bound in \eqref{eq:unifSf} is $O(N)$, uniformly over $p \leq N^c$. Thus, provided that $M \ra \infty$, Theorem \ref{thm:dil} asserts that if $f$ produces \emph{any} non-trivial bound as in \eqref{eq:unifSf} then $f$ must be equal to a character at \emph{almost all} $n < N$. 
\end{rem}
\begin{rem} \label{rem:smallS0}
If $g(p) = -1$ for at least one prime $p \leq N^c$ it is implicit in \eqref{eq:unifSf} (setting $a = 0$) that
\begin{equation}\label{eq:0freq}
\left|\sum_{n < N} f(n)\right| = O(N/\sqrt{M}).
\end{equation}
In this case, provided $M$ is sufficiently large this evidently means that $\psi = \chi$ in the conclusion of Theorem \ref{thm:dil}. By the prime number theorem, this case holds for $f$ the Liouville function and $M \leq \exp(\sqrt{\log N})$, say, which is why our Corollary \ref{cor:Liou} relates to a property of $L(s,\chi)$.
\end{rem}
\begin{rem} Note that Theorem \ref{thm:iff} is indeed a (conditional) strengthening of Theorem \ref{thm:dil}. In the latter, we assert that \emph{there exists} a set $\mc{A} = \{p \leq N^c\}$ of primes for which \eqref{eq:almostall} follows from \eqref{eq:genA}.
\end{rem}
\begin{rem}
It is reasonable to ask whether some of the hypotheses in Theorem \ref{thm:dil} (and indeed Theorem \ref{thm:iff} as well) can be relaxed. For instance, one could ask whether the same result holds for more general multiplicative functions $f: \mb{N} \ra S^1$, with other $S^1$-valued sequences $(g(p))_{p < N}$, or with $N$ not necessarily prime. We believe these to be interesting generalisations, but as our methods do not suffice to handle them\footnote{The problem appears to have substantial additional challenges when e.g. $f$ and $g$ are assumed to take values in $d$th order roots of unity, for some $d \geq 3$. In order to obtain a rigid conclusion similar to that of Theorem \ref{thm:dil}, we would need to show that there is at most one $0 \leq j \leq d-1$ such that $f(n)\chi(n)^j$ can have a ``large'' partial sum of length $N$, taking $\chi$ to be a primitive character modulo $N$ of order $d$. This could be ruled out provided we knew that $\chi(p)$ oscillated significantly along primes $p < N$. Unfortunately, current techniques do not allow one to show this unconditionally.}, we have chosen to defer their treatment to another occasion. 
\end{rem} 
\section{From primes to all residue classes $\pmod{N}$}
In the sequel, write $P^+(n)$ and $P^-(n)$ to denote, respectively, the largest and smallest prime factors of $n$ (with the conventions $P^+(1) = P^-(1) = 1$). \\
In this section, we extend estimates for 
$$
\max_{p \in S} \frac{1}{N} \sum_{a \pmod{N}} |S_f(ap/N) - g(p)S_f(a/N)|^2,
$$
wherein dilation is by \emph{primes} $p$ from some set $S$, to yield estimates
$$
\max_{b \in (\mb{Z}/N\mb{Z})^\times}\frac{1}{N}\sum_{a \pmod{N}} |S_f(ab/N) - w_b S_f(a/N)|^2
$$
for some sequence $(w_b)_{b \in (\mb{Z}/N\mb{Z})^\times} \in \{-1,+1\}^{N-1}$ (see Proposition \ref{prop:PbSmooth} below). We will exploit these latter estimates in the next section. \\
To address both of our theorems, we will separately concentrate on conditional and unconditional results in the next two subsections.
\subsection{Unconditional results}
In this subsection we will prove the following.
\begin{prop} \label{prop:uncond}
Let $c > \const$ and let $N$ be a prime chosen sufficiently large relative to $c$. Then there are $K_1,K_2 \geq 1$, depending only on $c$, such that
$$
\min_{a \in (\mb{Z}/q\mb{Z})^{\times}} |\{n < N^{K_1} : P^+(n) \leq N^c, \, \Omega(n) \leq K_2, \, n \equiv a \pmod{N}\}| \geq 1.
$$
\end{prop}
To this end we will use some ideas that are heavily inspired by a paper of Walker \cite{Walker}. In the sequel, given a finite abelian group $G$, a subset $S \subseteq G$ and $k \geq 1$ we write
$$
S^{(k)} := \{s_1\cdots s_k : \, s_j \in S \text{ for all } 1 \leq j \leq k\}.
$$ 
We will be interested in particular in the case $G = (\mb{Z}/q\mb{Z})^\times$. Given $\eta \in (0,1)$ we let $P_{\eta}$ denote the set of $a \in (\mb{Z}/q\mb{Z})^{\times}$ that are covered by primes $p < \eta q$. Then
$$
P_{\eta}^{(k)} = \{a \pmod{q} : \exists \, p_1,\ldots,p_k < \eta q, \, p_1p_2\cdots p_k \equiv a \pmod{q}\}.
$$
We prove the following result, which improves upon \cite[Thm. 2(ii)]{Walker}.
\begin{prop}\label{prop:impWalker}
Let $\delta > 0$, let $q$ be a prime chosen sufficiently large relative to $\delta$, and let $\eta = q^{-3/4+\delta}$. Then there is $K = O_{\delta}(1)$ such that $P_{\eta}^{(K)} = (\mb{Z}/q\mb{Z})^{\times}$. 
\end{prop}
The proof of this proposition follows a strategy of Walker \cite{Walker}, and will require several steps. First, we recall the following consequence of Burgess' theorem.
\begin{lem}\label{lem:Burgcor}
Let $\delta \in (0,1/4)$, and let $q$ be a prime that is sufficiently large in terms of $\delta$. Let $\xi$ be a non-principal character modulo $q$, and let $X \geq q^{1/4+\delta}$. Then
$$
\left|\sum_{n \leq X} \xi(n) \right| \ll_{\delta} Xq^{-\tfrac{1}{2}\delta^2}.
$$
\end{lem}
\begin{proof}
By Burgess' theorem (see e.g. \cite[Thm. 12.6]{IK}), for each $r \geq 2$ and $\e > 0$ we have
$$
\left|\sum_{n \leq X} \xi(n)\right| \ll_{\e} X^{1-1/r} q^{\tfrac{r+1}{4r^2} + \e} = X(q^{1/4}/X)^{1/r} q^{1/(4r^2) + \e} \leq Xq^{\tfrac{1}{r}\left(-\delta + \tfrac{1}{4r}\right) + \e}.
$$
We select $r = \lceil 1/\delta \rceil$, so that
$$
\frac{1}{r}\left(-\delta + \frac{1}{4r}\right) \leq -\frac{3\delta}{4r} \leq -\frac{3\delta^2}{4(1+\delta)} \leq -\frac{3\delta^2}{5},
$$
which implies the claim upon taking $0 < \e \leq 0.1 \delta^2$.
\end{proof}
\begin{lem} \label{lem:impWalker}
Let $\delta > 0$, let $q$ be a prime chosen sufficiently large in terms of $\delta$, and let $\eta \geq q^{-3/4+\delta}$. Then there is a $k_0 = k_0(\delta)$ such that if $k \geq k_0$ then $P_{\eta}^{(k)} \gg_{\delta,k} q$.
\end{lem}
\begin{proof}
Let $X := \eta q \geq q^{1/4+\delta}$. Observe that by Chebyshev's bounds,
$$
\sum_{a \in P_{\eta}^{(k)}} |\{p_1,\ldots,p_k < X : \, p_1\cdots p_k \equiv a \pmod{q}\}| = \pi(X-1)^k \gg_k \left(\frac{X}{\log X}\right)^k.
$$
Combining this with the Cauchy-Schwarz inequality, we have
\begin{align} \label{eq:Petabd}
\left(\frac{X}{\log X}\right)^{2k} \ll_k |P_{\eta}^{(k)}| |\{p_1,\ldots,p_{2k} < X : p_1\cdots p_k \equiv p_{k+1}\cdots p_{2k} \pmod{q}\}| =: |P_{\eta}^{(k)}| M_k.
\end{align}
Set $\alpha := \delta^2/8$ and define $z := q^{\alpha^2}$, $D := q^{\alpha}$. Let $(\lambda_d)_d$ be an upper bound beta sieve (see \cite[Sec. 6.2]{IK}) supported on squarefree integers $d \leq D$ with $p|d \Rightarrow p \leq z$, and let $w := 1\ast \lambda$. By orthogonality of characters, 
\begin{align*}
M_k &\leq \sum_{n_1,\ldots,n_{2k} < X} w(n_1)\cdots w(n_{2k}) 1_{n_1\cdots n_k \equiv n_{k+1}\cdots n_{2k} \pmod{q}} \\
&= \frac{1}{q-1} \left(\sum_{n < X}w(n)\right)^{2k} + \frac{1}{q-1} \sum_{\xi \neq \xi_0} \left|\sum_{n < X} w(n) \xi(n)\right|^{2k},
\end{align*}
where we have denoted by $\xi_0$ the principal character modulo $q$. If $\delta$ is small enough then by the fundamental lemma of the sieve \cite[Lem. 6.3]{IK}, the principal character contribution is
$$
\frac{1}{q-1} \left(\sum_{n < X}w(n)\right)^{2k}\ll_k \frac{1}{q} \left(\frac{X}{\log z}\right)^{2k} \leq \frac{1}{\alpha^{4k} q} \left(\frac{X}{\log X}\right)^{2k}.
$$
For each non-principal character $\xi$ modulo $q$ we expand the convolution of $w(n)$ and apply Lemma \ref{lem:Burgcor}, getting, 
$$
\sum_{n < X} w(n) \xi(n) = \sum_{d \leq D} \lambda_d \xi(d) \sum_{m < X/d} \xi(m) \ll D\max_{d \leq D} \frac{X}{d} q^{-\tfrac{1}{2}(\delta-\alpha)^2} \ll Xq^{-\tfrac{1}{2}\delta^2 + \alpha(1+\delta)} \leq Xq^{-\tfrac{1}{4}\delta^2}.
$$
It follows that
$$
M_k \ll \frac{1}{\alpha^{4k} q}\left(\frac{X}{\log X}\right)^{2k}  + X^{2k} q^{-\tfrac{1}{4}\delta^2 k}.
$$
We now select $k_0 = 1+ \lceil 4/\delta^2\rceil$, so that if $q$ is sufficiently large and $k \geq k_0$ then 
$$
X^{2k} q^{-\tfrac{1}{4}\delta^2k} \ll \frac{1}{q}\left(\frac{X}{\log X}\right)^{2k}.
$$
We therefore deduce that $M_k \ll_{\delta} \frac{1}{q}\left(\frac{X}{\log X}\right)^{2k}$, and conclude from \eqref{eq:Petabd} that $|P_{\eta}^{(k)}| \gg_{\delta,k} q$, as required.
%
\end{proof}
\noindent To pass from a set $P_{\eta}^{(k)}$ containing a positive proportion of non-zero residue classes modulo $q$ to one covering all of $(\mb{Z}/q\mb{Z})^{\times}$, we use the following well-known result of Freiman (see \cite[Lem. 8]{Walker}, and its application in \cite[Proof of Thm. 3]{Walker}).
\begin{lem} \label{lem:Frei}
Let $G$ be a finite abelian group, let $S \subseteq G$ be a set that is not contained in a coset of a proper subgroup of $G$. Then either $|S^{(2)}| \geq \frac{3}{2}|S|$ or $S^{(4)} = G$.
\end{lem}
\begin{proof}[Proof of Proposition \ref{prop:impWalker}]
By \cite[Lem. 9]{Walker}, we know that $P_{\eta}$ is not contained in a coset of a proper subgroup of $(\mb{Z}/q\mb{Z})^{\times}$, provided $\eta \geq q^{-3/4+\delta}$. Consequently, the same is true of $P_{\eta}^{(k)}$, for any $k \geq 1$ (as otherwise, fixing $p_0 < \eta q$ we could choose for each $p < \eta q$ products $P_1,P_2 \in P_{\eta}^{(k)} \subseteq x H$ such that $P_2/P_1 = p/p_0$, whence $p \in p_0 H$ for all $p < \eta q$, a contradiction). \\
As $\eta \geq q^{-3/4+\delta}$, by Proposition \ref{prop:impWalker} we can find $k = k(\delta)$ such that $P_{\eta}^{(k)} \geq c(\delta)q$ for some $c(\delta) > 0$ that depends only on $\delta$.  Applying Lemma \ref{lem:Frei} iteratively, we see that there is a $m \leq 1+\tfrac{\log(1/c(\delta))}{\log(3/2)}$ such that if $K := 4 \cdot 2^m k$ then $P_{\eta}^{(K)} = (\mb{Z}/q\mb{Z})^\times$, as claimed.
\end{proof}
\begin{proof}[Proof of Proposition \ref{prop:uncond}]
We apply the previous lemma with $q = N$, letting $\delta := c-1/4 > 0$. Then there is a $K_2 = K_2(\delta) > 0$ such that every $a \in (\mb{Z}/q\mb{Z})^{\times}$ may be represented as a product $p_1\cdots p_{K_2}$ with each $p_j < N^c$. But any such product has size $< N^{c K_2}$, so the claim follows upon taking $K_1 := cK_2$. 
\end{proof}

\subsection{Conditional results}
Our main task in this subsection is to prove the following result, which is our conditional analogue of Proposition \ref{prop:uncond}.
\begin{prop} \label{prop:bddOmega}
Let $A, \eta >0$, $c \in (0,1/2)$ such that $A$ is sufficiently large in terms of $c$ and $\eta$, and $\eta \leq 0.1 c^4/(\log(1/c))^2$. Let $N$ be a prime of $A$-Littlewood type, chosen sufficiently large in terms of $A$. Let $\mc{A} \subseteq \mb{P} \cap [2,N]$ be a set of primes satisfying
$$
\sum_{\ss{p < N \\ p \notin \mc{A}}} \frac{1}{p} \leq \eta.
$$
Then there is a $K = K(c)$ such that
$$
\min_{1 \leq a < N} |\{n \leq N^{1/c} : \, p|n \Rightarrow p \in \mc{A} \cap [2,N^c], \, \Omega(n) \leq K, \, n \equiv a \pmod{N}\}| \geq 1.
$$
\end{prop}
The above proposition will follow from the following more general statement.
\begin{prop}\label{prop:soundCond}
Let $\eta,A,C > 0$ with $A$ sufficiently large in terms of $\eta$ and $C$. Let $2 \leq z \leq y \leq x$, let $q$ be a prime modulus that is of $A$-Littlewood type and chosen sufficiently large in terms of $A$, and suppose $y \leq q \leq y^C$. Define
$$
u := \frac{\log x}{\log y}, \quad v := \frac{\log x}{\log z}, \quad w := \frac{\log x}{\log q}.
$$
Let $\mc{A} \subseteq [2,q] \cap \mb{P}$ be a set of primes with
\begin{equation}\label{eq:etacond}
\sum_{\ss{p \leq q \\ p \notin \mc{A}}} \frac{1}{p} \leq \eta,
\end{equation}
and assume that $u\log u \leq \eta^{-1/2}$. Then, provided that $u$ is sufficiently large in terms of $C$, and $v/u \geq (\log (2u))^2$, 
$$
\min_{1 \leq a < q} |\{n \leq x : \, n \equiv a \pmod{q}, \, p|n \Rightarrow p \in \mc{A} \cap (z,y]\}| 
\gg \frac{xu^{-2u}}{vq\log x}.
$$
\end{prop}
\noindent Let us introduce some further notation. For $1 \leq z \leq y \leq x$, $1 \leq q \leq x$ and $a \in (\mb{Z}/q\mb{Z})^{\times}$ define
\begin{align*}
\Theta(x,y,z;q,a) &:= |\{n \leq x : \, n \equiv a \pmod{q}, \, p|n \Rightarrow p \in \mb{P}\cap (z,y]\}|, \\
\Theta(x,y,z) &:= |\{n \leq x : \, p|n \Rightarrow p \in \mb{P} \cap (z,y]\}|.
\end{align*}
The study of $\Theta(x,y,z)$ was initiated by Friedlander in \cite{Fri}; definitive results in wide ranges of the parameters $x,y$ and $z$ were given in a series of papers by Saias \cite{Saias1}, \cite{Saias2} and \cite{Saias3}. As far as we are aware, however, there is comparatively little discussion of $\Theta(x,y,z;q,a)$ in the literature (certainly in the range $y \leq q$). \\
In the appendix to this paper we will use some of the literature on $\Theta(x,y,z)$ to adapt to our setting a method invented by Soundararajan \cite{Sound}, and later refined by Harper \cite{Harper}, to study \emph{$y$-friable} numbers (i.e., integers $n$ such that if $p|n$ then $p \leq y$) in arithmetic progressions. As it is not the primary focus of this work we only consider restricted ranges in $q,y$ and $z$, namely, when 
$$
\min\{\sqrt{x}, z^{\log\log x}\} \geq y \geq z^2, \quad z \geq \exp(\sqrt{\log x}), \quad y \leq q \leq y^C,
$$
and $u$ and $v/u$ (in the notation of Proposition \ref{prop:soundCond}) are sufficiently large; however, these results can probably be proved in a wider range of parameters. The key estimate that we will prove is Proposition \ref{prop:soundGen} below. \\
In what follows, we let $\Psi(x,y) = \Theta(x,y,1)$ denote the number of $y$-friable integers $n \leq x$, and define $\alpha = \alpha(x,y,z)$ to denote the unique real number $\sg$ such that
$$
\sum_{z < p \leq y} \frac{\log p}{p^{\sg}-1} = \log x.
$$
By \cite[Thm. 2]{Saias3}, it is known that under the above assumptions on $x,y,z$, we have 
\begin{equation}\label{eq:alphaEst}
\alpha = 1-\frac{\log(u\log u) + O(1)}{\log y}.
\end{equation}
\begin{prop} \label{prop:soundGen}
Let $A,C > 0$, let $1< z \leq y \leq x$, let $q$ be a prime modulus that is of $A$-Littlewood type, and assume $y \leq q \leq y^C$. Define
$$
u := \frac{\log x}{\log y}, \quad v := \frac{\log x}{\log z}, \quad w := \frac{\log x}{\log q}.
$$
Assume also that 
$$
\min\{\sqrt{x}, z^{\log\log x}\} \geq y \geq z^2, \quad z \geq \exp(\sqrt{\log x}),
$$ 
that $u \ra \infty$ as $x \ra \infty$ and that $v/u \geq (\log(2u))^2$. 
Then, 
uniformly over $1 \leq d \leq y$,
$$
\Theta(x/d,y,z;q,a) = (1+o_{u \ra \infty}(1)) \frac{\Psi(x,y)}{d^{\alpha}\phi(q)} \prod_{p \leq z}\left(1-\frac{1}{p}\right).
$$
\end{prop}
\begin{proof}[Proof of Proposition \ref{prop:soundCond} assuming Proposition \ref{prop:soundGen}]
Let $1 \leq a < q$. Note that
\begin{align*}
&|\{n \leq x : \, n \equiv a \pmod{q}, \, p|n \Rightarrow p \in \mc{A} \cap (z,y]\}|
\geq \Theta(x,y,z;q,a) - \sum_{\ss{z < p \leq y \\ p \notin \mc{A}}} \Theta(x/p,y,z; q,ap^{-1}).
\end{align*} 
By Proposition \ref{prop:soundGen}, the latter is
\begin{equation}\label{eq:approxThaq}
(1+o(1))\frac{\Psi(x,y)}{\phi(q)} \prod_{p \leq z} \left(1-\frac{1}{p}\right) \left(1- \sum_{\ss{ z < p \leq y \\ p \notin \mc{A}}} \frac{1}{p^{\alpha}} \right).
\end{equation}
Combining \eqref{eq:alphaEst}, \eqref{eq:etacond} and Mertens' theorems, we get
$$
(1+o(1))\frac{e^{-\gamma}\Psi(x,y)}{\phi(q) \log z} \left(1 - \sum_{\ss{ z < p \leq y \\ p \notin \mc{A}}} \frac{p^{1-\alpha}}{p} \right) = (1+o(1))\frac{e^{-\gamma}\Psi(x,y)}{\phi(q) \log z}\left(1-O(\eta u \log u)\right).
$$
By assumption we have $u\log u \leq \eta^{-1/2}$, so if $\eta$ is sufficiently small then
$$
|\{n \leq x : \, n \equiv a \pmod{q}, \, p|n \Rightarrow p \in \mc{A} \cap (z,y]\}|  \geq (1+o(1))\frac{e^{-\gamma}\Psi(x,y)}{2\phi(q) \log z},
$$
By the well-known lower bound $\Psi(x,y) \gg x u^{-(1+o(1))u}$ given in \cite{CEP} (valid in our range $y \geq \exp(\sqrt{\log x})$), the claim now follows. 
\end{proof}
\begin{proof}[Proof of Proposition \ref{prop:bddOmega} assuming Proposition \ref{prop:soundCond}]
Let $x = N^{1/c}$, $y = N^c$ and $z = N^{c^2}$, so that $u = c^{-2}$ and $v = c^{-3}$, and $y \leq N \leq y^{1/c}$. Assuming that $c^{-1} \geq \log(2/c^2)^2$, that $A$ is sufficiently large in terms of $c$ and $\eta$, and finally that $N$ is sufficiently large relative to $A$, Proposition \ref{prop:soundCond} yields
$$
|\{n \leq N^{1/c} : \, n \equiv a \pmod{N}, \, p|n \Rightarrow p \in \mc{A} \cap (N^{c^2},N^c]\}| \geq 1.
$$
But for any $n$ in this latter set,
$$
N^{1/c} \geq n \geq P^-(n)^{\Omega(n)} > N^{c^2\Omega(n)},
$$
so setting $K := c^{-3}$ we find that $\Omega(n) \leq K$, as claimed.
\end{proof}

\subsection{Dilation by arbitrary residue classes}
We apply Propositions \ref{prop:bddOmega} and \ref{prop:uncond} in order to prove Proposition \ref{prop:PbSmooth} below (the notation of which is the same as in the aforementioned two propositions). To this end, we extend the sequence $(g(p))_{p \leq N^c}$ by zeros and thence define a completely multiplicative function $g(n)$ in the obvious way.
\begin{prop} \label{prop:PbSmooth}
Let $c > \const$. For every non-zero $b \pmod{N}$ there is a positive integer $P_b \equiv b \pmod{N}$ satisfying the following properties: 
\begin{enumerate}[(i)]
\item $P_b < N^{2K_1}$, $\Omega(P_b) \leq 2K_2$ and $P^+(P_b) \leq N^{c}$,
\item if $b$ is a quadratic residue modulo $N$ then $g(P_{b}) = +1$, and
\item we have 
$$
\max_{b \in (\mb{Z}/N\mb{Z})^{\times}}\frac{1}{N} \sum_{a \pmod{N}} \left|S_{f}(ab) - g(P_b)S_{f}(a)\right|^2 \ll \frac{N}{M}.
$$
\end{enumerate}
Moreover, suppose we are given $\eta >0$ and $\mc{A} \subset \mb{P} \cap (0,N)$ with 
$$
\sum_{\ss{p < N \\ p \notin \mc{A}}} \frac{1}{p} \leq \eta.
$$ 
If $N$ is a large prime of $A$-Littlewood type then we may take any $c \in (0,1/2)$, and provided $\eta \leq 0.1 c^{4}/\log(2/c^2)^2$ and $A$ is sufficiently large in terms of $\eta$, then we may replace (i) above by
$$
P_b < N^{2/c}, \, \Omega(P_b) \leq 2K, \text{ and } p|P_b \Rightarrow p \in \mc{A} \cap [2,N^c].
$$
\end{prop}
\begin{proof}
First, let us show that if, for all non-zero $b \pmod{N}$ there exists $P_b \equiv b \pmod{N}$ such that condition (i) holds, then condition (iii) holds as well. In this case we can write $P_b = p_1\cdots p_k$ with $p_1 \leq p_2 \leq \cdots \leq p_k$ and $k = \Omega(P_b) \leq 2K_2 = O(1)$. Since $S_f(ba) = S_f(P_ba)$, telescoping and applying the Cauchy-Schwarz inequality, we have\footnote{Here, if $j =0$ then we interpret the (empty) product $p_1\cdots p_j$ as $1$.}
\begin{align*}
&\frac{1}{N}\sum_{a \pmod{N}} |S_{f}(ba) - g(P_b) S_{f}(a)|^2 \\
&= \frac{1}{N} \sum_{a \pmod{N}} \left|\sum_{0 \leq j \leq k-1} (-1)^{k-1-j}(S_{f}(p_1\cdots p_{j+1}a) - g(p_{j+1})S_{f}(p_1\cdots p_ja))\right|^2 \\
&\ll k \sum_{0 \leq j \leq k-1} \left(\frac{1}{N} \sum_{a \pmod{N}} \left|S_{f}(p_1\cdots p_{j+1}a) - g(p_{j+1})S_{f}(p_1\cdots p_ja)\right|^2\right) \\
&\ll k^2 \max_{0 \leq j \leq k-1} \left(\frac{1}{N} \sum_{d \pmod{N}} \left|S_{f}(p_{j+1}d) - g(p_{j+1})S_{f}(d)\right|^2\right) \\
&\ll \max_{p \leq N^{c}} \frac{1}{N} \sum_{d\pmod{N}} \left|S_{f}(pd) - g(p) S_{f}(d)\right|^2,
\end{align*}
making the invertible change of variables $d \equiv p_1\cdots p_j a \pmod{N}$ for each $0 \leq j \leq k-1$ in the second-to-last line, then taking the maximum over $j$. Since this bound holds uniformly over all $b \in (\mb{Z}/N\mb{Z})^{\times}$, applying \eqref{eq:unifSf} gives
$$
\max_{b \in (\mb{Z}/N\mb{Z})^{\times}} \frac{1}{N} \sum_{a \pmod{N}} |S_{f}(ba) - g(P_b) S_{f}(a)|^2 \ll \frac{N}{M}.
$$
Thus, provided the $P_b$ satisfy condition (i) then they also satisfy condition (iii). \\
We now show that a choice of $(P_b)_b$ exists for which both conditions (i) and (ii) hold. By Proposition \ref{prop:uncond}, for every $1 \leq b< N$ we can find $n_b \equiv b \pmod{N}$ with 
$$
P^+(n_b) \leq N^{c}, \, n_b < N^{K_1} \text{ and } \Omega(n_b) \leq K_2.
$$ 
If $b$ is a quadratic non-residue then we set $P_b := n_b$. \\
Next, suppose $b \equiv a^2 \pmod{N}$ for some $a \in (\mb{Z}/N\mb{Z})^{\times}$. We have that $n_a^2 \equiv b \pmod{N}$ and
$$
n_a^2 < N^{2K_1}, \, \Omega(n_a^2) = 2 \Omega(n_a) \leq 2K_2, \text{ and } P^+(n_a^2) \leq N^{c}. 
$$
Setting $P_b := n_a^2$, we necessarily have $g(P_b) = g(n_a^2) = +1$, and condition (ii) holds. Since for this choice of $(P_b)_b$ condition (i) must hold we get that condition (iii) holds, and the first statement of the lemma is proved. \\
The second statement follows similarly, applying Proposition \ref{prop:bddOmega} in place of Proposition \ref{prop:uncond}.
\end{proof}
\section{Proof of Theorem \ref{thm:dil}}
\subsection{From exponential sums to mean values}
We recall our notation $\chi(n) = \left(\tfrac{n}{N}\right)$, for $N$ a large prime. 
Using Proposition \ref{prop:PbSmooth} we establish the following key claim that will be of use in proving both our unconditional and conditional results.
\begin{prop}\label{prop:LiouCorrel}
Suppose that
\begin{equation}\label{eq:exttob}
\max_{b \in (\mb{Z}/N\mb{Z})^\times} \frac{1}{N} \sum_{a \pmod{N}} \left|S_{f}(ab) -g(P_b)S_{f}(a)\right|^2 = O(N/M),
\end{equation}
where the integers $P_b$ are as in Proposition \ref{prop:PbSmooth}. Then
provided $M$ is sufficiently large, there is a real character $\psi$ modulo $N$ such that
\begin{equation}
\frac{1}{N} \sum_{n < N} f(n) \psi(n) = 1 + O(M^{-1/2}).
\end{equation}
\end{prop}
We will require a few auxiliary results to prove this proposition. The first is a well-known theorem of Granville and Soundararajan \cite[Thm. 1]{GraSou}.
\begin{lem} \label{lem:spectrum}
There is a constant $\delta_1 > -2/3$ such that the following holds. For any multiplicative function $f : \mb{N} \ra [-1,1]$ we have
$$
\frac{1}{x} \sum_{n \leq x} f(n) \geq \delta_1 + o(1).
$$
\end{lem}
As a first instance of Proposition \ref{prop:LiouCorrel}, we resolve the case in which $\psi$ is the principal character modulo $N$ (i.e. $\psi(n) = 1$ for all $n < N$). 
\begin{lem} \label{lem:psiTriv}
Assume the hypotheses of Proposition \ref{prop:LiouCorrel}, and that $g(p) = +1$ for all $p \leq N^c$. Then, provided $M$ is sufficiently large,
$$
\frac{1}{N} \sum_{n < N} f(n) = 1 + O(M^{-1/2}).
$$
\end{lem}
\begin{proof}
Assume that $g(p) = +1$ for all $p \leq N^c$. As $P_b$ is $N^c$-friable, $g(P_b) = +1$ for all $b \in (\mb{Z}/N\mb{Z})^\times$. In that case, it follows from \eqref{eq:exttob} and Plancherel's theorem that for each $b \in (\mb{Z}/N\mb{Z})^\times$,
$$
\frac{|S_f(0)|^2}{N} + \frac{1}{N} \text{Re} \asum_{a \pmod{N}} S_f(ab) \bar{S_f}(a) = N + O(N/M).
$$
Averaging over $b \in (\mb{Z}/N\mb{Z})^{\times}$, we get
$$
\frac{|S_f(0)|^2}{N} + \frac{1}{N(N-1)}\left| \, \asum_{a \pmod{N}} S_f(a)\right|^2  = N + O(N/M).
$$
But as $N$ is prime we obviously have
$$
\asum_{a \pmod{N}} S_f(a) = \sum_{a \pmod{N}} S_f(a) - S_f(0) = -S_f(0).
$$
Combined with the previous estimate, we find that
$$
|S_f(0)|^2 = N(N-1)\left(1+O(1/M)\right) = N^2 + O(N + N^2/M).
$$
and as $1 \leq M \leq N$, this yields the estimate
$$
\left|\frac{1}{N}\sum_{n < N} f(n)\right| = 1+O(M^{-1/2}). 
$$
If $M$ is sufficiently large then it follows from Lemma \ref{lem:spectrum} that $\sum_{n < N} f(n) > 0$, and the claim follows.
\end{proof}
We will next show that if the hypotheses of Proposition \ref{prop:LiouCorrel} hold, and $g(p) = -1$ for at least a single prime $p \leq N^c$, then $S_f(a/N)/\sqrt{N}$ lies near the unit circle for typical $a \pmod{N}$, consistent with the behaviour of the twisted Gauss sums $S_{\chi}$. 
\begin{lem}\label{lem:eta}
Assume that the hypotheses of Proposition \ref{prop:LiouCorrel} hold, and that $g(p) = -1$ for some $p \leq N^c$. Then there is a $\tau \in S^1$ 
such that 
$$
\frac{1}{N} \asum_{a \pmod{N}} \left|S_{f}(a) - \tau g(P_a)\sqrt{N}\right|^2 \ll N/M.
$$
\end{lem}
\begin{proof}
If we restrict \eqref{eq:exttob} to non-zero classes $a \pmod{N}$, then average the corresponding sum over residue classes $1 \leq b < N$ then we get
$$
\frac{1}{N(N-1)} \asum_{a,b \pmod{N}} \left|S_{f}(ba)-g(P_b) S_{f}(a)\right|^2 \ll N/M.
$$
Swapping orders of summation and making the change of variables $d \equiv ab \pmod{N}$, we find
$$
\frac{1}{N-1} \asum_{d \pmod{N}} \left(\frac{1}{N} \asum_{a \pmod{N}} \left|S_{f}(d) - g(P_{da^{-1}}) S_{f}(a)\right|^2\right) \ll N/M,
$$
so that in particular we can find a $d \pmod{N}$ (which we fix once and for all) such that
\begin{equation}\label{eq:simulClose}
\frac{1}{N} \asum_{a \pmod{N}} |S_{f}(d) - g(P_{da^{-1}}) S_{f}(a)|^2 \ll N/M.
\end{equation}
From this, Plancherel's theorem, the Cauchy-Schwarz inequality and Remark \ref{rem:smallS0}, we see that
\begin{align*}
|S_{f}(d)|^2  &= \frac{1}{N} \sum_{a \pmod{N}} |S_{f}(a)|^2 + \frac{1}{N} \asum_{a \pmod{N}} |S_{f}(d) - g(P_{da^{-1}})S_{f}(a)|^2 \\
&+ O\left(\frac{|S_{f}(0)|^2}{N} + \frac{1}{N} \asum_{a \pmod{N}} |S_{f}(d) - g(P_{da^{-1}}) S_{f}(a)| |S_{f}(a)|\right) \\
&= N + O(N/\sqrt{M}).
\end{align*}
We deduce, therefore, that there is a $\theta = \theta(N) \in \mb{R}$ such that 
$$
S_{f}(d) = e(\theta) \sqrt{N} + O(N^{1/2}M^{-1/2}).
$$ 
Setting $\eta(a) := g(P_{da^{-1}}) \in \{-1,+1\}$ for convenience, we obtain from \eqref{eq:simulClose} that
\begin{equation} \label{eq:propii}
\frac{1}{N} \asum_{a \pmod{N}} |S_{f}(a) - e(\theta)\eta(a) \sqrt{N}|^2 \ll N/M.
\end{equation}
Next, for each $1 \leq b < N$ the change of variables $a\mapsto ab \pmod{N}$ in \eqref{eq:propii} leaves the sum over $a$ invariant. This, together with \eqref{eq:exttob}, yields the related bounds
\begin{align*}
&\frac{1}{N} \asum_{a \pmod{N}} |S_{f}(ab) - e(\theta)\eta(ab)\sqrt{N}|^2 \ll N/M, \\
&\frac{1}{N} \asum_{a \pmod{N}} |S_{f}(ab) - g(P_b) S_{f}(a)|^2 \ll N/M, \\
&\frac{1}{N} \asum_{a \pmod{N}} |S_{f}(a) - e(\theta)\eta(a) \sqrt{N}|^2 \ll N/M.
\end{align*}
It follows from this and the Cauchy-Schwarz inequality that for each non-zero $b \pmod{N}$,  
$$
\frac{1}{N}\asum_{a \pmod{N}} |\eta(ab) - g(P_b) \eta(a)|^2 \ll M^{-1}.
$$
Averaging over $b$ as well, we get
\begin{equation}\label{eq:etarel}
\frac{1}{N^2}\asum_{a,b \pmod{N}} |\eta(ab) - g(P_b) \eta(a)|^2 \ll M^{-1}.
\end{equation}
On the other hand, swapping roles of $a$ and $b$, we have by the same argument that
\begin{equation}\label{eq:etarelswap}
\frac{1}{N^2}\asum_{a,b \pmod{N}} |\eta(ab) - g(P_a) \eta(b)|^2 \ll M^{-1}.
\end{equation}
Combining \eqref{eq:etarel} and \eqref{eq:etarelswap} with Cauchy-Schwarz, we conclude that
$$
\frac{1}{N^2} \asum_{a,b \pmod{N}} |\eta(b)g(P_a) - \eta(a)g(P_b)|^2 \ll M^{-1}.
$$
We thus deduce that there is a choice of non-zero $b_0 \pmod{N}$ such that if $\eta_0 := g(P_{b_0}) \eta(b_0)$ then
$$
\frac{1}{N} \asum_{a \pmod{N}} |\eta(a) - \eta_0g(P_a)|^2 \ll M^{-1}.
$$
Combining this with \eqref{eq:propii} and setting $\tau := e(\theta) \eta_0 \in S^1$ yields the claim.
\end{proof}
Finally, we will need the following result of Hall and Tenenbaum \cite{HT}.
\begin{lem}\label{lem:HT}
There is an absolute constant $\kappa > 0$ such that the following holds. Let $f: \mb{N} \ra [-1,+1]$ be a multiplicative function. Then for any $x \geq 3$,
$$
\sum_{n \leq x} f(n) \ll x \exp\left(-\kappa \sum_{p \leq x} \frac{1-f(p)}{p}\right). 
$$
\end{lem}
\begin{proof}[Proof of Proposition \ref{prop:LiouCorrel}]
If $g(p) = +1$ for all $p \leq N^c$ then the claim follows from Lemma \ref{lem:psiTriv}. Therefore, it remains to consider when $g(p) = -1$ for some $p \leq N^c$. \\ 
Let $\mc{Q}$ denote the set of (non-zero) quadratic residues modulo $N$. By Proposition \ref{prop:PbSmooth} (ii) we have
\begin{equation}\label{eq:trivQR}
\frac{1}{N} \sum_{b \in \mc{Q}} g(P_b) = \frac{|\mc{Q}|}{N} = \frac{1}{2} + O\left(\frac{1}{N}\right).
\end{equation}
On the other hand, we have
\begin{align} \label{eq:QRs}
\frac{1}{N} \sum_{b \in \mc{Q}} g(P_b) &= \frac{1}{N} \asum_{a \pmod{N}} \frac{1+\chi(a)}{2} g(P_a) \nonumber\\
&= \frac{1}{2N} \left(\asum_{a \pmod{N}} g(P_a) + \asum_{a \pmod{N}} g(P_a)\chi(a)\right) \\
&=: \frac{1}{2N}(S_1 + S_2).
\end{align}
We first estimate $S_1$. By Lemma \ref{lem:eta}, Remark \ref{rem:smallS0} and Cauchy-Schwarz we may observe that
\begin{align*}
\frac{1}{N}\sum_{a \pmod{N}} S_{f}(a) &= \frac{\tau}{\sqrt{N}} \asum_{a \pmod{N}} g(P_a) + O\left(\frac{|S_{f}(0)|}{N} + \left(\frac{1}{N} \asum_{a \pmod N} |S_{f}(a) -\tau g(P_a)\sqrt{N}|^2\right)^{1/2}\right) \\
&= \frac{\tau}{\sqrt{N}} \asum_{a \pmod{N}} g(P_a) + O\left(N^{1/2}M^{-1/2}\right),
\end{align*}
whereas also
$$
\sum_{a \pmod{N}} S_{f}(a) = \sum_{1 \leq n < N} f(n) \left(\sum_{a \pmod{N}} e(an/N)\right) = 0.
$$
As $\tau \in S^1$, we deduce that
\begin{equation}\label{eq:etaOnly}
S_1 = \asum_{a \pmod{N}} g(P_a) \ll NM^{-1/2}.
\end{equation}
Next, set $u_{N} := S_{\chi}(1)/\sqrt{N} \in \{1,i\}$. On combining \eqref{eq:GaussRigreg}, Plancherel's theorem, Lemma \ref{lem:eta} and the Cauchy-Schwarz inequality,
\begin{align}
\sum_{n < N} f(n) \chi(n) &= \frac{1}{N} \sum_{a \pmod{N}} S_{f}(a) \bar{S}_{\chi}(a) = \frac{\bar{u}_{N}}{\sqrt{N}} \asum_{a \pmod{N}} \chi(a) S_{f}(a) \nonumber\\
&= \tau \bar{u}_{N} \asum_{a \pmod{N}} \chi(a) g(P_a) + O(NM^{-1/2}). \label{eq:etachi}
\end{align}
Combining \eqref{eq:trivQR}, \eqref{eq:QRs}, \eqref{eq:etaOnly} and \eqref{eq:etachi}, we deduce that
\begin{align*}
\sum_{n < N} f(n)\chi(n) = \tau \bar{u}_{N} \asum_{a \pmod{N}}(1+\chi(a))g(P_a) + O(NM^{-1/2}) &= 2\tau \bar{u}_{N} \sum_{b \in \mc{Q}} g(P_b) + O(NM^{-1/2}) \\
&= \tau \bar{u}_{N} N + O(NM^{-1/2}).
\end{align*}
Now, since the LHS is real and $\tau \bar{u}_{N} \in S^1$, it follows that $\text{Im}(\tau \bar{u}_N) = O(M^{-1/2})$, and so
$$
\min_{\sg \in \{-1,+1\}} |\tau \bar{u}_{N}-\sg| \ll M^{-1/2}.
$$ 
But since $f \chi$ is a real-valued multiplicative function, by Lemma \ref{lem:spectrum} we find that if $M$ is sufficiently large then $|\tau \bar{u}_{N} +1| \geq 1/3$. Hence, we deduce that 
$$
\frac{1}{N}\sum_{n < N} f(n)\chi(n) = 1 + O(M^{-1/2}),
$$
as claimed.
\end{proof}
\subsection{Multiplicative functions that are close at a fixed scale}
The following lemma, which will be needed to prove the second claim of Theorem \ref{thm:dil}, may be of independent interest.
\begin{lem} \label{lem:verySparse}
Fix $\delta > 0$ and let $2+\delta \leq T \leq N$. Suppose $f,h: \mb{N} \ra \{-1,+1\}$ are multiplicative functions such that
$$
|\{n \leq N : f(n) \neq h(n)\}| \leq N/T.
$$
Then we have
$$
\sum_{\ss{p \leq N \\ f(p) \neq h(p)}} \frac{1}{p} \ll_{\delta} \frac{1}{\sqrt{\min\{T, \log N\}}}.
$$
\end{lem}
\begin{rem}
Let us make three remarks here. \\
Firstly, we stress that a condition like $T \geq 2+\delta$ is necessary in this statement: if for example we allow $T < 2$ then we may take e.g. $f(n) = \lambda(n)$ and $h \equiv 1$, for which $f(p) \neq h(p)$ for all $p < N$ but with $f(n) = h(n)$ for $\sim (1/2+o(1))N \leq N/T$ solutions for $N$ sufficiently large. \\
Secondly, Lemma \ref{lem:verySparse} is most interesting when $T = T(N) \ra \infty$ with $N$, in which case the set of primes $p$ where $f(p) \neq h(p)$ is rather sparse. If, conversely, $T$ is bounded with $N$ then we may recover the conclusion trivially from e.g. Lemma \ref{lem:HT} (see \eqref{eq:closeTo1} and \eqref{eq:closeonPrimes} below). \\
Thirdly, we note that if $f$ and $h$ are fixed functions and $N$ is growing then some form of this lemma follows quickly from Delange's theorem (see e.g. \cite[Sec. III.4.2]{TenBook}). 
However, in at least one case to which this result is applied, $h = \chi$ is a character of conductor $N$, so a scale-dependent result of the kind proved here is essential for our purposes. 
\end{rem}
\begin{proof}
Observe first that $0 \leq 1-f(n)h(n) \leq 2$, a fact we will use frequently below. Now, the hypothesis we have implies that 
\begin{align}\label{eq:closeTo1}
0 \leq \frac{1}{N} \sum_{n \leq N} (1-f(n)h(n)) \leq \frac{2}{T}, \quad \text{ so that } \quad \frac{1}{N} \sum_{n \leq N} f(n)h(n) \geq 1 - \frac{2}{T} \gg_{\delta} 1,
\end{align}
and in particular the sum is positive. Since $f(n)h(n)$ is multiplicative, by Lemma \ref{lem:HT} there is a constant $B = B(\delta) > 0$ such that
\begin{equation}\label{eq:closeonPrimes}
\sum_{p \leq N} \frac{1-f(p)h(p)}{p}  \leq B.
\end{equation}
Now let $\tfrac{1}{\sqrt{\log N}} \leq \eps < 1/2$ be a parameter to be chosen later. Define also $\mc{B} := \{p \leq N: f(p) \neq h(p)\}$ and $\mc{P} := \{p \leq N : p \notin \mc{B}\}$, as well as
$$
\Omega_{\mc{B}}(n) := \sum_{\ss{p^k || n \\ p \in \mc{B}}} k.
$$
On the one hand, by positivity of $1-f(n)h(n)$ we get
\begin{align*}
\sum_{n \leq N} (1-f(n)h(n))\Omega_{\mc{B}}(n) \geq \sum_{\ss{pm \leq N \\ p \in \mc{B} \cap [1,N^{1-\eps}] \\ q|m \Rightarrow q \in \mc{P}}} \left(1 - f(p)h(p) f(m)h(m)\right).
\end{align*}
Note that if $q|m \Rightarrow q \in \mc{P}$ then $f(m) = h(m)$. Since $1-f(p)h(p) \neq 0$ if and only if $p \in \mc{B}$, 
\begin{align} \label{eq:lowBd}
\sum_{n \leq N} (1-f(n)h(n))\Omega_{\mc{B}}(n) \geq \sum_{\ss{p \leq N^{1-\eps}}} (1-f(p)h(p)) |\{m \leq N/p : \, q|m \Rightarrow q \in \mc{P}\}|.
\end{align}
Next, define $v := e^{B/2 + 10}$. 
Combining \eqref{eq:closeonPrimes} with Mertens' theorem, we get that for every $N^{\eps} \leq X \leq N$,
$$
\sum_{\ss{p \in \mc{P} \\ X^{1/v} < p \leq X}} \frac{1}{p} \geq \sum_{X^{1/v} < p \leq X} \frac{1}{p} - \frac{1}{2} \sum_{p \leq N} \frac{1-f(p)h(p)}{p} \geq \log v - \frac{B}{2} + O\left(\frac{v}{\sqrt{\log N}}\right) \geq 2,
$$
say, provided $N$ is sufficiently large. By a theorem of Matom\"{a}ki and Shao \cite{WTSW2} (building on the earlier work \cite{WTSW}), we see that uniformly over $N^{\eps} \leq X \leq N$,
$$
|\{m \leq X : \, q | m \Rightarrow q \in \mc{P}\}| \gg v^{-v(1+o(1))}  \prod_{q \in \mc{B}} \left(1-\frac{1}{q}\right) X \gg v^{-v(1+o(1))} \exp\left(-\frac{1}{2}\sum_{q \leq N} \frac{1-f(q)h(q)}{q}\right)\geq c(B) X,
$$
for some constant $c(B) > 0$. Inserting this bound into \eqref{eq:lowBd}, we get
\begin{align}\label{eq:lowerOm}
\sum_{n \leq N} (1-f(n)h(n))\Omega_{\mc{B}}(n) 
\geq c(B) N \sum_{p \leq N^{1-\eps}} \frac{1-f(p)h(p)}{p}.
\end{align}
We next get an upper bound for the LHS in \eqref{eq:lowBd}. Let $K \geq 1$ be a parameter to be chosen later. 
Splitting $n \leq N$ according to whether or not $\Omega_{\mc{B}}(n) \leq K$, and using $0 \leq 1-f(n)h(n) \leq 2$ together with \eqref{eq:closeTo1}, we find 
\begin{align*}
\sum_{n \leq N} (1-f(n)h(n))\Omega_{\mc{B}}(n) &\leq K \sum_{n \leq N} (1-f(n)h(n)) + 2\sum_{\ss{n \leq N \\ \Omega_{\mc{B}} > K}} \Omega_{\mc{B}}(n) \leq \frac{2KN}{T} + \frac{2}{K}\sum_{\ss{n \leq N}} \Omega_{\mc{B}}(n)^2.
\end{align*}
To bound the sum on the RHS just above, we use the elementary estimate
\begin{align*}
\sum_{n \leq N} \Omega_{\mc{B}}(n)^2 = \sum_{\ss{p,q \leq N \\ p,q \in \mc{B}}} \sum_{\ss{n \leq N \\ [p,q] | n}} 1 = \sum_{\ss{[p,q] \leq N \\ p,q \in \mc{B}}} \left(\frac{N}{[p,q]} + O(1)\right) &\leq N\left(1+\sum_{p \in \mc{B}} \frac{1}{p}\right)^2 + O(|\{pq \leq N: \, p,q \in \mb{P}\}|) \\
&\leq (B+1)^2 N + O\left(\frac{N\log\log N}{\log N}\right),
\end{align*}
invoking \eqref{eq:closeonPrimes} in the last estimate. It follows that when $N$ is sufficiently large,
$$
\sum_{n \leq N} (1-f(n)h(n))\Omega_{\mc{B}}(n) \leq N\left(K/T + 3(B+1)^2/K\right).
$$
Combining this with \eqref{eq:lowerOm}, we get
$$
\sum_{p \leq N^{1-\eps}} \frac{1-f(p)h(p)}{p} \leq \frac{1}{c(B)} \left(\frac{K}{T} + \frac{3(B+1)^2}{K}\right).
$$
By Mertens' theorem, we may recombine the primes $N^{1-\eps} < p \leq N$, getting
$$
\sum_{p \leq N} \frac{1-f(p)h(p)}{p} \leq \frac{1}{c(B)} \left(\frac{K}{T} + \frac{3(B+1)^2}{K}\right) + 2\eps+ O\left(\frac{1}{\log N}\right).
$$
We choose $K = T^{1/2}$ and $\eps = (\log N)^{-1/2}$ to get
$$
\sum_{p \leq N} \frac{1-f(p)h(p)}{p} \ll_B \frac{1}{\sqrt{\min\{T, \log N\}}},
$$
and since $B$ depended only on $\delta$, the claim follows.
\end{proof}
\begin{proof}[Proof of Theorem \ref{thm:dil}]
Assume that $f$ satisfies \eqref{eq:unifSf}, with $1 \leq M \leq N$. We begin by proving the first claim. Since this claim is trivial otherwise, we may assume that $M$ is larger than any fixed, absolute constant. 
Combining the first statement of Proposition \ref{prop:PbSmooth} with Proposition \ref{prop:LiouCorrel}, we obtain that there is a real character $\psi$ modulo $N$ such that
$$
\frac{1}{N}\sum_{n < N} f(n)\psi(n) = 1+O(M^{-1/2}),
$$
and the first claim follows from the identity $1_{f(n)\neq \psi(n)} = (1-f(n)\psi(n))/2$, for $n < N$. The second claim follows from the first by Lemma \ref{lem:verySparse}, taking $T \asymp \sqrt{M}$. \\
Next, suppose $p \leq N^c$ is a prime for which $f(p) = \psi(p)$. We wish to show that $g(p) = f(p)$ in this case. For every $a \pmod{N}$ we of course have
$$
S_{\psi}(ap) = \psi(p) S_{\psi}(a) = f(p)S_{\psi}(a).
$$
Moreover, by Plancherel's theorem modulo $N$, we have
$$
1 + O(M^{-1/2}) = \frac{1}{N} \sum_{n < N} f(n) \psi(n) = \frac{1}{N^2} \sum_{a \pmod{N}} S_f(a)\bar{S_{\psi}}(a).
$$
We therefore deduce that
$$
\frac{1}{N^2} \sum_{a \pmod{N}} |S_f(a)-S_{\psi}(a)|^2 = \frac{1}{N^2} \sum_{a \pmod{N}} \left(|S_f(a)|^2 + |S_{\psi}(a)|^2 - 2\text{Re}(S_f(a) \bar{S_{\psi}}(a))\right) = O(M^{-1/2}).
$$
Making the invertible change of variables $a \mapsto ap \pmod{N}$, we thus also have
$$
\frac{1}{N^2} \sum_{a \pmod{N}} |S_f(ap)-S_{\psi}(ap)|^2 \ll M^{-1/2},
$$
so that by the Cauchy-Schwarz inequality we get
\begin{align}
&\frac{1}{N} \sum_{a \pmod{N}} |S_f(ap) - f(p)S_f(a)|^2 \nonumber\\
&\ll \frac{1}{N} \sum_{a \pmod{N}}\left(|S_f(ap) - S_{\psi}(ap)|^2 + |S_{\psi}(ap) - f(p)S_{\psi}(a)|^2 + |S_{\psi}(a)-S_f(a)|^2\right) \ll N/\sqrt{M}. \label{eq:Sfapconv}
\end{align}
A final application of Plancherel's theorem, together with \eqref{eq:unifSf} and the previous estimate, then gives
\begin{align*}
|f(p)-g(p)|^2N &= \frac{1}{N}\sum_{a \pmod{N}} |f(p)S_f(a) - g(p)S_f(a)|^2 \\
&\ll \frac{1}{N} \sum_{a \pmod{N}} \left(|S_f(ap)-f(p)S_f(a)|^2 + |S_f(ap)-g(p)S_f(a)|^2\right) \\
&\ll N/\sqrt{M}.
\end{align*}
If $M$ is large enough then as $f(p),g(p) \in \{-1,+1\}$ we find $f(p) = g(p)$, as claimed.
\end{proof}
\subsection{Proofs of Theorem \ref{thm:iff} and Corollary \ref{cor:Liou}}
\begin{proof}[Proof of Theorem \ref{thm:iff}]
The proof of part (b) is identical to the proof of Theorem \ref{thm:dil}, save that we combine Proposition \ref{prop:LiouCorrel} with the second statement, rather than the first, of Proposition \ref{prop:PbSmooth}. \\
The proof of part (a) is similar to proof of the third claim of Theorem \ref{thm:dil}. Assume that $f: \mb{N} \ra \{-1,+1\}$ is a multiplicative function for which there is a large prime $N$ such that \eqref{eq:almostall} holds for some real character $\psi$ modulo $N$. Set $\mc{B}_f := \{p < N : f(p)\neq \psi(p)\}$.  By Lemma \ref{lem:verySparse}, we have that
$$
\sum_{p \in \mc{B}_f} \frac{1}{p} \ll \frac{1}{\sqrt{\min\{M,\log N\}}},
$$
and (ii) is proved. The proof of (i) is the same as the proof of \eqref{eq:Sfapconv}, but with $M$ in place of $M^{1/2}$.
\end{proof}
We now give the proof of Corollary \ref{cor:Liou}.
\begin{proof}[Proof of Corollary \ref{cor:Liou}]
If $N$ is an exceptional modulus (in the sense e.g. of \cite[Thm. 5.26]{IK}) then the claim that a \emph{real} zero $\beta \in (0,1)$ exists such that $(1-\beta)\log N \leq c_0$ for some fixed $c_0$, is immediate. Thus, we may assume in the remainder of the argument that $N$ is non-exceptional.\\
Assume that $\lambda$ satisfies \eqref{eq:unifSf}. As before, write $\chi = (\tfrac{\cdot}{N})$. By Theorem \ref{thm:dil} and the prime number theorem (to rule out $\psi$ being principal) we get that
$$
\frac{1}{N}\sum_{n < N} \lambda(n)\chi(n) = 1+o(1).
$$
Applying Lemma \ref{lem:HT} with $f = \lambda$, we get
$$
\sum_{p < N} \frac{\chi(p)}{p} = -\log\log N + O(1).
$$ 
As $N$ is non-exceptional, by e.g., \cite[(4.3)]{GraMan} we obtain
$$
L(1,\chi) \asymp \prod_{p < N} \left(1-\frac{\chi(p)}{p}\right)^{-1} \asymp \exp\left(\sum_{p < N} \frac{\chi(p)}{p}\right) \asymp \frac{1}{\log N}.
$$
Finally, by a classical result of Hecke (as discussed e.g. in \cite[Rem. 1.7]{MaMe}), we deduce that there is a real zero $\beta$ of $L(s,\chi)$ such that $(1-\beta) \ll (\log N)^{-1}$ as well, which finishes the proof.
\end{proof}

\section*{Appendix: Integers free of large and small prime factors in arithmetic progressions}
In this appendix we will prove Proposition \ref{prop:soundGen}. To proceed further we will recall and introduce further notation. Given $1 < z \leq y \leq x$, define 
$$
u := \frac{\log x}{\log y}, \quad v := \frac{\log x}{\log z}.
$$
For $\text{Re}(s) >0$, $q \geq 1$ and a character $\chi \pmod{q}$ we set
$$
L(s,\chi;y,z) := \prod_{z < p \leq y} \left(1-\frac{\chi(p)}{p^{s}}\right)^{-1}, \quad \zeta(s,y,z) := \prod_{z < p \leq y} \left(1-\frac{1}{p^{s}}\right)^{-1}.
$$
Important to the study of 
$$
\Theta(x,y,z) := |\{n \leq x : p|n \Rightarrow z < p \leq y\}|
$$ 
is the parameter $\alpha = \alpha(x,y,z)$, mentioned in relation to Proposition \ref{prop:soundGen}. We recall that $\alpha$ is the unique real solution to the equation
$$
\sum_{z < p \leq y} \frac{\log p}{p^{\alpha}-1} - \log x = 0.
$$
We will assume throughout that 
\begin{equation}\label{eq:assumps}
\min\{\sqrt{x}, z^{\log\log x}\} \geq y \geq z^2, \quad z \geq \exp(\sqrt{\log x}), \quad y \leq q \leq y^C,
\end{equation}
that $u$ is sufficiently large in terms of $C$, and that $v/u \geq (\log u)^2$. By \cite[Thm. 2]{Saias3}, under these conditions we have
$$
\alpha = 1-\frac{\log(u\log u) +O(1)}{\log y}.
$$
We will need the following properties of $\Theta(x,y,z)$, and the related $\Psi(x,y) = \Theta(x,y,1)$.
\begin{lem}\label{lem:SaiasAsymp}
Let $2 \leq z \leq y \leq x$ with $y \geq z^2$ and $z \geq \exp(\sqrt{\log x})$, and let $v$ and $u$ be defined as above. Then provided that $u \ra \infty$ and $v/u \geq (\log (2u))^2$,
$$
\Theta(x,y,z) = (1+o(1)) \Psi(x,y) \prod_{p \leq z} \left(1-\frac{1}{p}\right).
$$
More generally, provided $u \ra \infty$, $y \ra \infty$ and $y \geq 2z$,
\begin{equation}\label{eq:ThetaBd}
\Theta(x,y,z) \asymp \frac{x^{\alpha} \zeta(\alpha,y,z)}{\sqrt{\log x \log y}}
\end{equation}
\end{lem}
\begin{proof}
The first statement follows immediately from \cite[Thm. B]{Saias1} (note that there is a misprint in the range of the Euler product, which is corrected in \cite[p. 357]{Saias3}). The second follows from \cite[Thm. 1-2]{Saias3}.
\end{proof}
\begin{lem} \label{lem:dlBTen}
Let $(\log x)^2 \leq y \leq x$. Then, uniformly over real numbers $1 \leq d \leq y^2$,
$$
\Psi(x/d,y) = (1+O(u^{-1}))d^{-\alpha} \Psi(x,y).
$$
Moreover, uniformly in the range $1 \leq d \leq x$ we have $\Psi(x/d,y) \ll d^{-\alpha}\Psi(x,y)$.
\end{lem}
\begin{proof}
The first claim is immediate from \cite[Thm. 2.4(ii)]{dlBTen}, taking $m = 1$ and noting that $\bar{u} = u \leq u_y$ and $t \leq 2$ in the notation there; the second claim is immediate from \cite[Thm. 2.4(i)]{dlBTen}.
\end{proof}
To obtain Proposition \ref{prop:soundGen} we will adapt a method that was invented by Soundararajan \cite{Sound}, and refined by Harper \cite{Harper}, to study $y$-friable integers in arithmetic progressions modulo $q$, when $y \leq q \leq y^C$ for some $C > 0$. In the sequel, we write
$$
\mc{T}_{y,z} := \{n\in \mb{N} : p|n \Rightarrow z < p \leq y\}.
$$
Let $0 \leq \Phi \leq 1_{[0,1]}$ be a fixed smooth function. We introduce the notation
$$
\Theta(x,y,z;q,a;\Phi) := \sum_{\ss{n \in \mc{T}_{y,z} \\ n \equiv a \pmod{q}}} \Phi(n/x).
$$
By orthogonality of Dirichlet characters, we have
\begin{equation}\label{eq:orthoChar}
\Theta(x,y,z;q,a;\Phi) = \frac{1}{\phi(q)} \sum_{\ss{n \in \mc{T}_{y,z}}} \Phi(n/x) 1_{(n,q) = 1} + \frac{1}{\phi(q)} \sum_{\chi \neq \chi_0} \bar{\chi}(a) \sum_{n \in \mc{T}_{y,z}} \chi(n) \Phi(n/x).
\end{equation}
We arrange the non-principal characters into sets as follows. For each $0 \leq k \leq \log q/2$ we let
$$
\mc{B}(k) := \{s = \sg + it : \, \sg \geq 1-k/\log q, \, |t| \leq q\},
$$
and define
$$
\Xi(k) := \{\chi \neq \chi_0 : \, L(s,\chi) \neq 0 \text{ for all } s \in \mc{B}(k), \, L(s,\chi) = 0 \text{ for some } s \in \mc{B}(k+1)\}.
$$
By the log-free zero density estimate \cite[p. 428]{IK}, there are constants $C_1,C_2 > 0$ such that $|\Xi(k)| \leq C_1 e^{C_2k}$. We assume here that $q$ is of $A$-Littlewood type, so that if we set $K := A\log\log q$ then\footnote{This assumption allows us to eschew the analysis of the ``Rodosskii'' and ``problem'' ranges that arise in \cite{Sound} and \cite{Harper}. Working a bit harder, one could attempt to extend those ranges as well, however when $u$ or $v/u$ does not grow sufficiently quickly the method of Harper that ought to give the better range $c > 1/(4\sqrt{e})$ does not seem to extend. See Remark \ref{rem:Harperext} below for a discussion of this issue.}
$$
\Xi(k) \neq \emptyset \Rightarrow k \geq K.
$$
Therefore, let $K \leq k \leq \log q/2$. For $\chi \in \Xi(k)$, Mellin inversion yields
$$
\sum_{n \in \mc{T}_{y,z}} \chi(n) \Phi(n/x) = \frac{1}{2\pi i} \int_{\alpha - i \infty}^{\alpha + i\infty} x^s \check{\Phi}(s) L(s,\chi;y,z) ds,
$$ 
where we have written $\check{\Phi}(s) := \int_0^{\infty} \Phi(x) x^{s-1} dx$ to denote the Mellin transform of $\Phi$. 
By the usual integration by parts argument, we have $|\check{\Phi}(s)| \ll_{m,\Phi} (1+|s|)^{-m}$ for all $m > 0$.\\
By a slight modification of \cite[Prop. 1]{Harper}, we can prove the following.
\begin{lem}[cf. Prop. 1 of \cite{Harper}] \label{lem:prop1}
Fix $B > 0$, let $K = A\log\log q \leq k \leq (\log q)/2$ and let $\chi \in \Xi(k)$. Assume that $z \leq y \leq q \leq x$ satisfy \eqref{eq:assumps} with $u \ra \infty$ as $x \ra \infty$, and that $A$ is sufficiently large relative to $C$. Let $x/y \leq X \leq x$ and let $\alpha' := \alpha - Bk/\log X$. Then, provided $x$ is sufficiently large,
$$
\sup_{\ss{\alpha' \leq \sg \leq \alpha \\ |t| \leq q/2}} \left|\log L(\sg + it,\chi;y,z) - \log L(\alpha + it,\chi;y,z)\right| \leq k/50.
$$
\end{lem}
\begin{proof}
Let $\alpha' \leq \sg \leq \alpha$ and $|t| \leq q/2$. We of course have
\begin{align*}
\left|\log L(\sg + it,\chi;y,z) - \log L(\alpha + it,\chi;y,z)\right| &\leq 2 \max_{z \leq r \leq y} |\log L(\sg+it,\chi;r) - \log L(\alpha + it,\chi;r) |,
\end{align*}
writing $L(s,\chi;r) := L(s,\chi;r,1)$. Now, fix $z \leq r \leq y$ for the time being. We have (as in \cite[p. 196]{Harper})
$$
\sup_{\alpha' \leq \sg \leq \alpha}|\log L(\sg+it,\chi;r) - \log L(\alpha + it,\chi;r)| \leq \frac{Bk}{\log X} \left(\sup_{\alpha' \leq \sg \leq \alpha} \left|\sum_{n \leq r} \frac{\Lambda(n)\chi(n)}{n^{\sg+it}} \right| + O(1)\right).
$$
We define $R = R(r) := \max\{2,r^{r^{-k/(2\log q)}}\}$ so that $2 \leq R \leq r$, and define the weight function $w = w_R$ by
$$
w(n) := \begin{cases} 1 &\text{ if } 1 \leq n \leq r \\ 1-\tfrac{\log(n/r)}{\log R} & \text{ if } r < n \leq Rr  \\ 0 &\text{ otherwise.} \end{cases} 
$$
Following the argument of \cite{Harper}, 
$$
\sum_{n\leq Rr} w(n)\frac{\Lambda(n)\chi(n)}{n^{\sg+it}} \ll \log q + \frac{1}{\log R} \left|\sum_{\ss{L(\rho,\chi) = 0 \\ 0 < \text{Re}(\rho) < 1}}\frac{(Rr)^{\rho-\sg-it} - r^{\rho - \sg-it}}{(\rho-\sg-it)^2}\right|.
$$
As for all non-trivial zeros $\rho$ of $L(s,\chi)$ we have $\text{Re}(\rho) - \sg < 1-\alpha'$, and as $C\log y \geq \log q \geq \sqrt{\log x}$ and $k \geq A \log\log q$, we have
$$
\alpha' = 1-\frac{\log(u\log u)+O(1)}{\log y} - \frac{Bk}{\log X} \geq 1-\frac{k}{\log q}\left(\frac{C}{A} + \frac{B}{u-1} + o(1)\right) \geq 1-\frac{k}{4\log q}
$$ 
provided $u$ is sufficiently large, and $A$ is large enough relative to $C$. Using standard bounds on the number of zeros of $L(s,\chi)$ with $|\text{Im}(\rho)| \leq q$, the above is bounded by
$$
\ll r^{-3k/(4\log q)} \frac{(\log q)^2}{k \log R} + \frac{1}{\log R} \sum_{|\text{Im}(\rho)| > q} \frac{(Rr)^{k/(4\log q)}}{|\text{Im}(\rho)|^2}  \ll \frac{(\log q)^2}{k\log r} + \frac{y^{1/4} \log q}{q} \ll \frac{(\log q)^2}{k\log r}.
$$
Following Harper's argument verbatim (noting that $r \leq y \leq q$ for us), the difference between the weighted and unweighted sums is
$$
\left|\sum_{r < n \leq Rr} \frac{\Lambda(n)w(n)\chi(n)}{n^{\sg+it}}\right| \ll \log q.
$$
We deduce therefore that
\begin{align*}
\sup_{\alpha' \leq \sg \leq \alpha}|\log L(\sg+it,\chi;r) - \log L(\alpha + it,\chi;r)| \ll \frac{Bk}{\log X} \left( \log q + \frac{(\log q)^2}{k \log r}\right).
\end{align*}
By \eqref{eq:assumps}, we have
$$
\frac{B\log q}{\log X} + \frac{B(\log q)^2}{k (\log r)\log X} \leq \frac{BC}{u-1} + \frac{BC^2 \log y}{K(u-1) \log z} \leq\frac{2BC}{u}\left(1+ \frac{C \log\log x}{A \log\log q}\right) \ra 0
$$
as $x\ra \infty$, the claim follows with $x$ sufficiently large.
\end{proof}
We also need the following results. 
\begin{lem}\label{lem:coll}
(a) Let $T \leq x^{1/4}$ and let $\sg > 0$. Then
$$
\sup_{0 \leq t \leq T} \left|\int_t^T x^{ir} \check{\Phi}(\sg + ir) dr \right| \ll \frac{1}{\log x}.
$$
(b) (Majorant principle) Let $(a_n)_n,(b_n)_n$ be complex-valued sequences with $|a_n| \leq b_n$ for all $n$. Let $T \geq 0$ and $N \geq 1$. Then
$$
\int_{-T}^T \left|\sum_{1 \leq n \leq N} a_n n^{-it}\right|^2 dt \leq 3\int_{-T}^T \left|\sum_{1 \leq n \leq N} b_n n^{-it}\right|^2 dt. 
$$
(c) Assume that $z \geq (\log q)^2$ and $T \leq (\log z)^{100}$. For every $1/2 < \sg \leq \alpha$ we have
$$
\int_{-T}^T \left|\frac{L'(\sg + it,\chi;y,z)}{L(\sg+it,\chi;y,z)}\right|^2 dt \ll y^{2(\alpha - \sg)} u^2 (\log u) \log y.
$$
(d) There is a constant $c > 0$ such that the following holds: for $|t| \leq y$ we have
\begin{align*}
\left|\frac{L(\alpha + it,\chi_0;y,z)}{L(\alpha,\chi_0;y,z)}\right| \ll \begin{cases} e^{-cu(t\log y)^2} & \text{ if } |t|\log y \leq 1, \\ \exp\left(-cu \frac{(t\log y)^2}{(t\log y)^2 + \log(u\log u)^2}\right) &\text{ if } |t| \log y > 1. \end{cases}
\end{align*}
(e) Let $1 \leq T \leq q/2$. Then, under the assumptions of Lemma \ref{lem:prop1}, we have
$$
\int_{-T}^T |L(\sg+it,\chi;y,z)|^2 dt \ll e^{k/25} L(\alpha,\chi_0;y,z)^2 \left(\frac{1}{\log y} + Te^{-2cu}\right),
$$
for some absolute constant $c > 0$.
\end{lem}
\begin{rem} \label{rem:Harperext}
Note that in (c), the factor $\log y$ arises from the larger primes in the support of $(L'/L)(s,\chi;y,z)$. In its application in Lemma \ref{lem:intSg}, this factor causes a loss in precision unless it can be compensated by the bound in (e). However, this is only possible provided $u$ is sufficiently large (specifically, of size $\gg \log\log y$). The reason for this is that there is not enough oscillation in $L(s,\chi;y,z)$ to give cancellation of size $\log y$ (since $\alpha$ is close enough to $1$, the maximal amount of possible cancellation in prime sums is approximately $\log(\tfrac{\log y}{\log z})$, which in our applications in this paper are an arbitrarily slowly growing function of $x$).  For this reason, we are unable to extend the unconditional argument of Harper to obtain $q^{1/(4\sqrt{e})+\e}$-smooth numbers with few prime factors in arithmetic progressions modulo $q$, and therefore also to extend our range $c>1/4$ to $c > 1/(4\sqrt{e})$ in Theorem \ref{thm:dil}.
\end{rem}
\begin{proof}
(a) This is immediate from \cite[App. B]{Harper}. \\
(b) This can be found e.g. in \cite[Ch. III.4]{TenBook}. \\
(c) Since $\sigma > 1/2$ and $z \geq \exp(\sqrt{\log q})$,
$$
\frac{L'(s,\chi_0;y,z)}{L(s,\chi_0;y,z)} = \sum_{z < n \leq y} \frac{\Lambda(n)}{n^s} + O\left(\frac{1}{z^{1/4}}\right).
$$
By partial summation and the prime number theorem, we have
\begin{align*}
\sum_{z < n \leq y} \frac{\Lambda(n)}{n^{\sg+it}}  = \frac{y^{1-\sg - it} - z^{1-\sg-it}}{1-\sg-it} + O\left((1+|t|) y^{1-\sg} e^{-\sqrt{\log z}}\right).
\end{align*}
As $\sg \leq \alpha$ and $y^{1-\sg} \asymp y^{\alpha-\sg} u\log u$, for all $|t| \leq T$, we get
$$
\left|\sum_{z < n \leq y} \frac{\Lambda(n)}{n^{\sg+it}}\right| \ll y^{\alpha-\sg} u(\log u)\left( \min\left\{\frac{1}{|t|}, \frac{\log y}{\log u}\right\} + (1+|t|) e^{-\sqrt{\log z}}\right).
$$
Squaring both sides and integrating over $|t| \leq T$, we get the bound
\begin{align*}
\int_{-T}^T \left|\frac{L'(\sg + it,\chi_0;y,z)}{L(\sg+it,\chi_0;y,z)}\right|^2 dt \ll y^{2(\alpha - \sg)} (u \log u)^2 \left(\frac{\log y}{\log u} + T^3 e^{-2\sqrt{\log z}}\right).
\end{align*}
The claimed bound (for $\chi$, rather than $\chi_0$) now follows upon applying (b) (taking $a_n := \chi(n)\Lambda(n)n^{-\sg}$ and $b_n := \chi_0(n) \Lambda(n)n^{-\sg}$). \\
(d) This is extracted from \cite[Lem. 10]{Saias3}. \\
(e) By Lemma \ref{lem:prop1}, as long as $\alpha' \leq \sg \leq \alpha$ and $T\leq q/2$ then we get
$$
\int_{-T}^T |L(\sg+it,\chi;y,z)|^2 dt \ll e^{k/25} \int_{-T}^T |L(\alpha + it,\chi;y,z)|^2 dt.
$$
We now apply (b) to get
$$
\int_{-T}^T |L(\alpha+it,\chi;y,z)|^2 dt \ll \int_{-T}^T |L(\alpha+it,\chi_0;y,z)|^2 dt = L(\alpha,\chi_0;y,z)^2 \int_{-T}^T \left|\frac{L(\alpha+it,\chi_0;y,z)}{L(\alpha,\chi_0;y,z)}\right|^2 dt.
$$
Splitting the range of $|t| \leq T$ into the ranges $|t| \leq 1/\log y$, $1 < |t|\log y \leq \log(u\log u)$ and $\log(u\log u)< |t|\log y \leq T\log y$, we get by (c) that
$$
\int_{-T}^T \left|\frac{L(\alpha+it,\chi_0;y,z)}{L(\alpha,\chi_0;y,z)}\right|^2 dt \ll \frac{1}{\log y} + T e^{-2cu},
$$
and the claim follows.
\end{proof}
The parts of the previous lemma combine to yield the following.
\begin{lem} \label{lem:intSg}
Fix $B > 0$. Let $x,y,z$ and $q$ satisfy \eqref{eq:assumps} with $u$ sufficiently large, and let $x/y \leq X \leq x$.  Let $K \leq k \leq (\log q)/2$ and suppose $\chi \in \Xi(k)$. Finally, put $\alpha' := \alpha - Bk/\log X$, let $\alpha' \leq \sg \leq \alpha$ and $1 \leq T \leq (\log z)^{100}$.  Then
$$
\left|\int_{\sg - iT}^{\sg+iT} X^s \check{\Phi}(s) L(s,\chi;y,z) ds \right| \ll e^{k/50} X^{\sg-\alpha}\Theta(x,y,z) \sqrt{u\log u}\left(1 + \sqrt{T\log y} e^{-cu}\right).
$$
\end{lem}
\begin{proof}
On combining \cite[Lem. 1]{Harper} (taking $G(s)$ to be the Euler product of the function $g(p) = 1_{(z,y]}(p)$ on primes, and $F(s) = X^s\check{\Phi}(s)$) with Lemma \ref{lem:coll}(a), we get
\begin{align*}
&\left|\int_{\sg - iT}^{\sg+iT} X^s \check{\Phi}(s) L(s,\chi;y,z) ds \right| \\
&\ll \frac{X^{\sg}}{\log x} \left(\sup_{0 \leq t \leq T} |L(\sg+it,\chi;y,z)| + \left(\int_{-T}^T \left|\frac{L'(\sg+it,\chi;y,z)}{L(\sg+it,\chi;y,z)}\right|^2 dt\right)^{1/2} \left(\int_{-T}^T |L(\sg+it,\chi;y,z)|^2 dt\right)^{1/2}\right).
\end{align*}
By Lemma \ref{lem:prop1} and the triangle inequality, we have $|L(\sg+it,\chi;y,z)|\leq e^{k/50} L(\alpha,\chi_0;y,z)$. 
Combining (c) and (e) of Lemma \ref{lem:coll}, the previous estimate is
$$
\ll e^{k/50}\frac{X^{\sigma} L(\alpha,\chi_0;y,z)}{\log X} \left(1 + u \sqrt{\log u \log y}\left(\frac{1}{\sqrt{\log y}} + \sqrt{T}e^{-cu}\right)\right).
$$
As $u \geq 10$ and $L(\alpha,\chi_0;y,z) \leq \zeta(\alpha,y,z)$, inserting \eqref{eq:ThetaBd} into this last estimate implies the claim.
\end{proof}
\begin{proof}[Proof of Proposition \ref{prop:soundGen}]
Let $1 \leq d \leq y$, and as before we write $u := (\log x)/(\log y)$. Recalling the definition of $\Xi(k)$ and the bound $|\Xi(k)| \leq C_1 e^{C_2k}$ for $K \leq k \leq (\log q)/2$, we may rewrite \eqref{eq:orthoChar} as
\begin{align}\label{eq:expink}
\Theta(x/d,y,z;q,a; \Phi) = \frac{1}{\phi(q)} \sum_{n \in \mc{T}_{y,z}} \Phi(dn/x) 1_{(n,q) = 1} + O\left(\frac{1}{\phi(q)}\sum_{K \leq k \leq (\log q)/2} e^{C_2k} \max_{\chi \in \Xi(k)} \left|\sum_{n \in \mc{T}_{y,z}} \chi(n)\Phi(dn/x)\right|\right). 
\end{align}
Fix $K \leq k \leq (\log q)/2$ and $\chi \in \Xi(k)$ for the moment. Applying Mellin inversion along the line $\text{Re}(s) = \alpha$ and using the rapid decay of $\check{\Phi}$, we may truncate at height $|\text{Im}(s)| = q^{1/2}$ with error 
$$
\sum_{n \in \mc{T}_{y,z}} \chi(n)\Phi(nd/x) = \frac{1}{2\pi i} \int_{\alpha - i\sqrt{q}}^{\alpha+i\sqrt{q}} (x/d)^s \check{\Phi}(s) L(\alpha + it, \chi;y,z) dt + O\left(\frac{(x/d)^{\alpha} L(\alpha,\chi_0;y,z)}{q^{B+10}}\right).
$$
We apply the residue theorem to shift the line of integration to $\text{Re}(s) = \alpha' = \alpha - Bk/\log(x/d)$. By Lemma \ref{lem:prop1},
the horizontal lines contribute
$$
\ll_B q^{-(B+10)} (x/d)^{\alpha} \max_{\alpha' \leq \sg \leq \alpha} |L(\sg,\chi;y,z)| \ll e^{k/50} q^{-(B+10)} (x/d)^{\alpha} \zeta(\alpha, y,z) \ll e^{-k(B-1/50)} d^{-\alpha} \Theta(x,y,z).
$$
Let $T = (\log z)^{10}$; since $z \geq \exp(\sqrt{\log x})$, we have $T \geq (\log x)^5$. We further truncate the integral along $\text{Re}(s) = \alpha'$ to $|\text{Im}(s)| = T$. Using the rapid decay of $\check{\Phi}$ once again, together with Lemmas \ref{lem:SaiasAsymp} and \ref{lem:prop1}, the contribution from $|t| \geq T$ is
$$
\ll_C e^{-k(B-1/50)}(\log x)^{-10}(x/d)^{\alpha}\zeta(\alpha,y,z) \ll e^{-k(B-1/50)} d^{-\alpha}\Theta(x,y,z).
$$
Using Lemma \ref{lem:intSg} to estimate the integral along $\text{Re}(s) = \alpha'$, we get 
$$
\sum_{n \in \mc{T}_{y,z}} \chi(n)\Phi(nd/x) \ll e^{-k(B-1/50)} d^{-\alpha}\Theta(x,y,z) \sqrt{u\log u} \left(1+e^{-cu} \sqrt{T\log y}\right). 
$$
Since this holds for every $\chi \in \Xi(k)$ and each $K \leq k \leq (\log q)/2$, inserting this into \eqref{eq:expink} yields the error term contribution
\begin{align*}
&\ll \frac{\Theta(x,y,z)}{d^{\alpha}\phi(q)} \left(\sqrt{u\log u} + T\sqrt{\log y} e^{-cu/2}\right)\sum_{K \leq k \leq (\log q)/2} e^{C_2k-k(B-1/50)}\\
&\ll \frac{\Theta(x,y,z)}{d^{\alpha}\phi(q)} e^{-2K}\left(\sqrt{u\log u} + T\sqrt{\log y} e^{-cu/2}\right), 
\end{align*}
provided that $B \geq C_2 + 3$. Since $T \leq (\log y)^{10}$, $u \leq \log x$ and $K \geq A \log\log q \geq \tfrac{1}{2}A \log\log x$, if $A$ is sufficiently large this yields
\begin{align}\label{eq:MTTodo}
\Theta(x,y,z;q,a; \Phi) = \frac{1}{\phi(q)} \sum_{n \in \mc{T}_{y,z}} \Phi(dn/x) 1_{(n,q) = 1} + O\left(\frac{\Theta(x,y,z)}{d^{\alpha}\phi(q) (\log x)^{10}}\right). 
\end{align}
We now remove the smoothing $\Phi$. Let $\e = \e(x) \ra 0$ sufficiently slowly. Taking $0 \leq \Phi \leq 1_{[0,1]}$ to be a smooth deformation of $1_{[0,1]}$ that is $1$ on $[0,1-\e]$, and recalling that $q$ is prime, we get
\begin{align*}
\sum_{n \in \mc{T}_{y,z}} \Phi(dn/x) 1_{(n,q) = 1} &= \sum_{\ss{n \leq x/d \\ n \in \mc{T}_{y,z}}} 1_{(n,q) = 1} + O\left(\sum_{\ss{(1-\e)x/d <n \leq x/d \\ n \in \mc{T}_{y,z}}} 1\right) \\
&=  \Theta(x/d,y,z) + O(\Theta(x/(dq),y,z) + |\Theta(x/d,y,z) - \Theta((1-\e)x/d,y,z)|).
\end{align*}
As $1 \leq d \leq y$, upon employing Lemma \ref{lem:SaiasAsymp} and Lemma \ref{lem:dlBTen} we get that 
\begin{align*}
\Theta(x/d,y,z) = (1+o_{u \ra \infty}(1))\prod_{p \leq z}\left(1-\frac{1}{p}\right) \Psi(x/d,y)
&= (1+o_{u \ra \infty}(1)) \frac{1}{d^{\alpha}}\Psi(x,y) \prod_{p \leq z} \left(1-\frac{1}{p}\right) \\
&= (1+o_{u \ra \infty}(1)) d^{-\alpha}\Theta(x,y,z).
\end{align*}
The error term may be treated similarly.
First,
$$
\Theta(x/(dq),y,z) \leq \Psi(x/(dq),y) \ll (dq)^{-\alpha} \Psi(x,y) \ll \frac{\log z}{\sqrt{q}} d^{-\alpha}\Psi(x,y)\prod_{p \leq z} \left(1-\frac{1}{p}\right) \ll \frac{\log y}{\sqrt{y}} d^{-\alpha}\Theta(x,y,z),
$$
which is acceptable. Next, as $(1-\e)d^{-1} \geq (2d)^{-1}$ for $x$ large enough,
\begin{align*}
\Theta(x/d,y,z) - \Theta((1-\e)x/d,y,z) &\ll \prod_{p \leq z} \left(1-\frac{1}{p}\right) \left|\Psi(x/d,y) - \Psi((1-\e)x/d,y)\right| \\
&\ll  d^{-\alpha} \Psi(x,y) \prod_{p \leq z} \left(1-\frac{1}{p}\right) \left(1-(1-\e)^{\alpha}\right) \\
&\ll \e d^{-\alpha}\Theta(x,y,z).
\end{align*}
Combining these estimates, 
we find that
$$
\sum_{n \in \mc{T}_{y,z}} \Phi(dn/x) 1_{(n,q) = 1} = (1+o_{u \ra \infty}(1))d^{-\alpha}\Theta(x,y,z) + O\left(\left(\e + \frac{\log y}{\sqrt{y}}\right) d^{-\alpha}\Theta(x,y,z)\right) = (1+o_{u \ra \infty}(1)) d^{-\alpha} \Theta(x,y,z). 
$$
Combined with \eqref{eq:MTTodo}, we find
$$
\Theta(x/d,y,z;q,a; \Phi) = (1+o_{u \ra \infty}(1))\frac{\Theta(x,y,z)}{d^{\alpha}\phi(q)},
$$
uniformly over $1 \leq d \leq y$. Combining this with Lemma \ref{lem:SaiasAsymp} suffices to prove the claim once $u$ is sufficiently large.
\end{proof}

\bibliographystyle{amsplain}

\end{document}